\documentclass[12pt,a4paper]{article}
\usepackage{amsmath,amsthm,amsfonts,amssymb}

 \newcommand{\bea}{\begin{eqnarray}}

\usepackage{graphicx,latexsym}
\usepackage{lscape}

\newtheorem{thm}{Theorem}[section]
 
 \newtheorem{lem}[thm]{Lemma}

\newtheorem{rem}[thm]{Remarks}

\newtheorem{discu}{Discussion:}

\newtheorem{conje}{Conjecture:}

\begin{document}
\begin{center}
\vspace{1cm}
 {\bf \large Moore-Penrose inverses of  Gram matrices Leaving a Cone Invariant in an Indefinite Inner Product Space}
 \vskip 1cm
  {\bf K. Appi Reddy, T. Kurmayya}\\
                        Department of Mathematics\\
                        National Institute of Technology Warangal\\
                                Warangal - 506 004, India. \\
\end{center}

\begin{abstract}
   In this paper we characterize Moore-Penrose inverses of Gram matrices leaving a cone invariant in an indefinite inner product space using indefinite matrix
multiplication.  This characterization includes  the acuteness (or
obtuseness) of
 certain closed convex cones.
\end{abstract}

\vspace{1cm} {\bf Keywords:}\, \, Gram matrix; Moore-Penrose
inverse; acute cones; Indefinite inner product space. \, \,

\vspace{.2cm} {\bf AMS Subject Classification:}\,\,\,\, 46C20, 15A09.\\
\thispagestyle{empty}

\thispagestyle{empty}

\newpage
\section{Introduction}

An indefinite inner product in $\mathbb{C}^{n}$ is a conjugate
symmetric sesquilinear form $[x, y]$ together with the regularity
condition that $[x, y]=0$ $ \forall y \in \mathbb{C}^{n}$ holds only
when $x = 0$.  Associated with any indefinite inner product, there
exists a unique invertible hermitian matrix $N\in
\mathbb{C}^{n\times n}$ (called a weight) such that $[x, y] =
\langle x, Ny\rangle, $ where $\langle.,.\rangle$ denotes the
Euclidean inner product on $\mathbb{C}^{n}$ and vice versa.
Motivated by the notion of Minkowski space (as studied by
physicists),  we also make an additional assumption on N, namely,
$N^{2} = I$.  It should be remarked that this assumption also allows
us to compare our results with the Euclidean case,  apart from
allowing us to present the results with much algebraic ease.

Investigations of linear maps on indefinite inner product spaces
employ the usual multiplication of matrices which is induced by the
Euclidean inner product of vectors (See for instance \cite{jb}).
This causes a problem as there are two different values for the dot
product of vectors.  To overcome  this difficulty;  Kamaraj,
Ramanathan and Sivakumar  introduced a new matrix product called
indefinite matrix multiplication and investigated some of its
properties in \cite{krkkkcs}.  More precisely,  the indefinite
matrix product of two matrices A and B of sizes $m \times n$ and $n
\times l$ complex matrices,  respectively,  is defined to be the
matrix $A\circ B := ANB$. The adjoint of A,  denoted by $A^{[*]}$,
is defined to be the matrix $NA^{*}M$,  where N and M are weights in
the appropriate spaces. Many properties of this product are similar
to that of the usual matrix product (refer \cite{krkkkcs}).
Moreover,  it not only rectifies the difficulty indicated earlier,
but also enables us to recover some interesting results in
indefinite inner product spaces in a manner analogous to that of the
Euclidean case. Kamaraj, Ramanathan and Sivakumar  \cite{krkkkcs}
also shown that in the setting of indefinite inner product spaces,
Moore-Penrose inverses of certain matrices do not exist with respect
to the usual matrix product where as Moore-Penrose inverses of such
matrices exist with respect to the indefinite matrix product. Hence
they concluded that indefinite matrix product is more appropriate
than the usual matrix product. %Recall that in an indefinite inner

     The problem of nonnegative invertibility of matrices (or inverses of  matrices leaving a cone invariant)  was first studied by Collatz \cite{cz2} when he
applied a finite difference method for solving a class of two point
boundary value problems.  This idea of nonnegative invertibility has
undergone
 a plethora of generalizations over the years. We refer the reader \cite{bp2} (and the references cited there in)
 for a detailed survey of these extensions.

 In recent years, nonnegative invertibility of Gram matrices has received, a lot
of attention. This has been primarily motivated by applications in
convex optimization problems. In this connection, there is a well
known result that characterizes non negative invertibility of Gram
matrices in terms of obtuseness or acuteness of certain polyhedral
cones. (See for instance Lemma 1.6 in \cite{ce}).
%However, Novikoff
%\cite{novi}, seems to have been the first to have proposed the
%notion of an obtuse cone. He characterized the nonnegativity of
%inverse of Gram operators.
Recently,  Sivakumar \cite{s2gram} characterized Moore-Penrose
inverses of Gram operators leaving a cone invariant over Hilbert
spaces. In this paper, we follow the approach of Sivakumar
\cite{s2gram} and discuss the  Moore-Penrose inverses of Gram
matrices leaving a cone invariant in an indefinite inner product
space using indefinite matrix product. As the indefinite matrix
product encompasses the Euclidean case as a particular example, it
follows that earlier results in the finite dimensional Euclidean
spaces, are easy corollaries of our main result.

The paper is organized as follows. In section 2, we introduce basic
notations, definitions and results. In section 3, we prove series of
lammas and derive the main theorem.

\section{Notations,  Definitions and Preliminaries}

 In this section, we  introduce notations, definitions and basic results  that will be used in
the rest of the paper.

Let $\langle.,.\rangle$ denote the usual Euclidean inner product in
$\mathbb{R}^{n}$. An indefinite inner product is denoted by
$[x,y]=\langle x,Ny\rangle$, where $N\in \mathbb{R}^{n\times n}$ and
$N=N^{-1}$. Such a matrix $N$ is called weight. A space with an
indefinite inner product is called an indefinite inner product
space. In the rest of the paper $\mathbb{R}^{m}, \mathbb{R}^{n}$
denote indefinite inner product spaces with weights $M,N$
respectively. Let $A,B$ be two real matrices of order $m \times n$
and $n \times l$ respectively, then the indefinite matrix product of
those matrices be denoted by $A\circ B$ and defined as $A\circ
B=ANB$, where $N$ is a weight matrix as defined earlier. For $A\in
\mathbb{R}^{m\times n}$, the adjoint $A^{[*]}$,  of $A$ is defined
by $A^{[*]}=NA^{*}M$, where * denotes the transpose of $A$, $M$ and
$N$ are weights of order m and n respectively.

 Let $K$ be a subset
of $\mathbb{R}^n$.  Then $K$ is called cone if (i) $x, y\in
K\Rightarrow x+y\in K$ and  (ii)$x\in K$, $\alpha\in \mathbb{R}$,
$\alpha\geq0\Rightarrow \alpha x\in K$. The dual of cone  $K$ is
denoted by $K^{[*]}$ and is defined as $K^{[*]}=\{x\in
\mathbb{R}^n:[x, t]\geq0,~ \text{for all}~ t\in K\}$. Let
$K^{[*][*]}$ denote $({K^{[*]}})^{[*]}$. If $K=\mathbb{R}^n_{+}$
then $K^{[*]}=I\circ \mathbb{R}^n_{+}$ and $K^{[*][*]}=K$.

A cone $C$ is said to be acute if $[x, y]\geq0$ for all $x, y\in C$.
$C$ is said to be obtuse if $C^{[*]}\cap\{cl~span~C\}$ is acute. In
particular, let $C=A\circ I\circ K$ then we say that $C=\{A\circ
I\circ x:x\in K\}$ is obtuse if $(A\circ I\circ
K)^{[*]}\cap\mathcal{R}(A\circ I)$ is acute. According to Novikoff,
the acuteness of a cone $C$ in $\mathbb{R}^n$ is defined by the
inclusion $C\subseteq C^{*}$. We can easily verify this condition in
indefinite inner product spaces as $C\subseteq C^{[*]}$.

For $A\in \mathbb{R}^{m\times n}$,  $A^{[*]}\circ A$
will be called the Gram matrix of $A$. For $A\in \mathbb{R}^{m\times n}$, the following equations are known to have unique solution \cite{krkkkcs}:\\
$A\circ X\circ  A=A$, $X\circ  A\circ  X=X$, $(A\circ
X)^{[*]}=A\circ X$, $(X\circ  A)^{[*]}=X\circ  A$. Such an $X$ will
be denoted by $A^{[\dagger]}$. If the weight matrices in indefinite
inner product spaces are equal to identity then
$A^{[\dagger]}=A^{\dagger}$. We refer the reader \cite{bg} (and the
references cited there in)
 for a detailed  study of
$A^{\dagger}$.

Next, we collect some properties of $A^{[\dagger]}$. Some of these
have been proved in \cite{krkkkcs} and rest can be demonstrated
easily. The range space of $A$, $\mathcal{R}(A)$ is defined by
$\mathcal{R}(A)=\{y\in \mathbb{R}^m: y=A\circ  x, x\in
\mathbb{R}^n\}$ and the null space  of $A$, $\mathcal{N}(A)$ is
defined by $\mathcal{N}(A)=\{x\in\mathbb{R}^n: A\circ x=0\}$. For
$A\in \mathbb{R}^{m\times n}$,
 $A^{[\dagger]}$ also satisfies the following properties:
 $\mathcal{R}(A^{[*]})=\mathcal{R}(A^{[\dagger]})$,
$\mathcal{N}(A^{[*]}) =\mathcal{N}(A^{[\dagger]})$, $A\circ
A^{[\dagger]}=P_{\mathcal{R}(A)}$, $A^{[\dagger]}\circ
A=P_{\mathcal{R}(A^{[*]})}$. We also have $(A^{[*]}\circ
A)^{[\dagger]}\circ (A^{[*]}\circ A)=P_{\mathcal{R}(A^{[*]}\circ
A)^{[\dagger]}}=P_{\mathcal{R}(A^{[*]}\circ
A)^{[*]}}=P_{\mathcal{R}(A^{[*]})} =A^{[\dagger]}\circ A$.
\begin{lem}
Let $A\in \mathbb{C}^{m\times n}$. Then\\
(i)$A^{[*]}=A^{[*]}\circ A\circ A^{[\dagger]}=A^{[\dagger]}\circ A\circ A^{[*]}$\\
(ii)$A^{[\dagger]}=A^{[*]}\circ (A\circ A^{[*]})^{[\dagger]}=(A^{[*]}\circ A)^{[\dagger]}\circ A^{[*]}$\\
(iii)$A^{[\dagger]}\circ (A^{[\dagger]})^{[*]}=(A^{[*]}\circ A)^{[\dagger]}$\\
(iv)$(A\circ I)^{[\dagger]}=I\circ A^{[\dagger]}$\\
(v)$\mathcal{R}(A\circ A^{[\dagger]})=\mathcal{R}(A)$,
$\mathcal{R}(A^{[\dagger]}\circ A)=\mathcal{R}(A^{[*]})$,
$\mathcal{N}(A\circ A^{[\dagger]})=\mathcal{N}(A^{[*]})$,
$\mathcal{N}(A^{[\dagger]}\circ A)=\mathcal{N}(A)$ where
$\mathcal{R}(X)$ and $\mathcal{N}(X)$  denote the range and null
spaces of X respectively.
\end{lem}
%\begin{lem}
%Let $A\in \mathbb{R}^{m\times n}$, then $[A\circ x, y]=[x, (I\circ A\circ I)^{[*]}\circ y]$ for all $x\in \mathbb{R}^n, y\in \mathbb{R}^m$\\
%\end{lem}
We use the following lemma frequently in this paper.
\begin{lem}\label{lineq}
Let $A\in \mathbb{R}^{m\times n}$ and $b\in \mathbb{R}^m$. Then, the
linear equation $A\circ X=b$ has a solution iff $b\in \mathcal{R}(A)
$. In this case, the general solution is given by
$x=A^{[\dagger]}\circ b+z$ where $z\in \mathcal{N(A)}$.
\end{lem}
%\begin{thm}
%Let $K$ be a subset of $\mathbb{R}^n$. Then $K^{[*]}=NK^{*}=I\circ K^{[*]}$. \\
%\end{thm}
%\begin{lem}
%Let $K=\mathbb{R}^n_{+}$ then $K^{[*]}=I\circ \mathbb{R}^n_{+}$ and $K^{[*][*]}=K$. \\
%\end{lem}
%
% In this paper we consider   matrices A such that$I_{m}\circ A=A\circ I_{n}$ that is  $I\circ A=A\circ I$ that is $MA=AN$ . \\
\section{Main Results}
   For given $A\in \mathbb{R}^{m\times n}$,  Ramanathan and Sivakumar
\cite{krksgram} derived a set of necessary and sufficient conditions
for a cone to be invariant under $(A^{[*]}\circ A)^{[\dagger]}.$
These conditions include pairwise acuteness (or pairwise obtuseness)
of certain cones. In this article, we avoid pairwise acuteness of
cones and characterize Moore-Penrose inverses of Gram matrices
leaving a cone invariant in the approach of Sivakumar \cite{s2gram}.
These results generalize the existing results of Sivakumar
\cite{s2gram} in the finite dimensional setting from Euclidean
spaces to indefinite inner product spaces. First we prove series of
lemmas that lead up to the main theorem (Theorem \ref{main}).

As mentioned earlier;  $\mathbb{R}^m, \mathbb{R}^n$ denote
indefinite inner product spaces with weights $M,N$ respectively. Let
$A\in\mathbb{R}^{m\times n}$ be such that $I\circ A=A\circ I$ that
is $MA=AN$ and let $K$ be a closed cone in $\mathbb{R}^n$.
\begin{lem}\label{lem1}
$[A\circ x, y]=[x, A^{[*]}\circ y]$.
\end{lem}
\begin{proof}
$[A\circ x, y]=\langle A\circ x, My\rangle=\langle ANx,
My\rangle=\langle x, NA^{*}My\rangle=[x, A^{*}My]= [x, I\circ
(NA^{*}M)\circ I\circ y]=[x, I\circ A^{[*]}\circ I\circ y]=[x,
(I\circ A\circ I)^{[*]}\circ y]=[x, A^{[*]}\circ y]. $
\end{proof}
\begin{lem}\label{lemma2}
 $u\in (A\circ I\circ K)^{[*]}\Rightarrow(A\circ I)^{[*]}\circ u\in {K}^{[*]}$.
\end{lem}
\begin{proof}
Let $u\in (A\circ I\circ K)^{[*]}$ and $r\in {K}$.  Then $0\leq [u,
A\circ I\circ r]=[(A\circ I)^{[*]}\circ u, r]$,   by Lemma
 \ref{lem1}.
Thus $(A\circ I)^{[*]}\circ u\in K^{[*]}$.
\end{proof}
Next, we show that $K$ is invariant under $A^{[\dagger]}\circ A$ if
and only if $K^{[*]}$ is invariant under $A^{[\dagger]}\circ A$.
\begin{lem}\label{AKdualAK}
$A^{[\dagger]}\circ A\circ K\subseteq K \Leftrightarrow
A^{[\dagger]}\circ A\circ K^{[*]}\subseteq K^{[*]}$.
\end{lem}
\begin{proof}
 Let $A^{[\dagger]}\circ A\circ K\subseteq K$,  $y=A^{[\dagger]}\circ A\circ x$ with $x\in
 K^{[*]}$,
 $u\in K ~\text{and}~u^{1}=A^{[\dagger]}\circ A\circ u\in K$.  Then
 $[y, u]=[A^{[\dagger]}\circ A\circ x, u]=[x, (A^{[\dagger]}\circ A)^{[*]}\circ u]=[x, A^{[\dagger]}\circ A\circ u]=[x, u^{1}]\geq0$.
 This shows that $y\in K^{[*]}$.  Hence $A^{[\dagger]}\circ A\circ K^{[*]}\subseteq K^{[*]}$.

 Similarly one can easily prove the converse part.
\end{proof}
In the next result,   we determine the set $(A\circ I\circ K)^{[*]}$
in the presence of an additional condition.
\begin{thm}\label{dual of Aok}
$(A\circ I\circ K)^{[*]}\subseteq(A^{[\dagger]})^{[*]}\circ I\circ
K^{[*]}+\mathcal{N}((A\circ I)^{[*]})$.
 If $K$ is invariant under $A^{[\dagger]}\circ A$,  then equality holds.
\end{thm}
\begin{proof}
 Let $y\in (A\circ I\circ K)^{[*]}$. Then by Lemma $\ref{lemma2}$, $z=(A\circ I)^{[*]}\circ y\in {K}^{[*]}. $
By Lemma \ref{lineq},  $y=((A\circ I)^{[*]})^{[\dagger]}\circ z+w$
for some $w\in \mathcal{N}((A\circ I)^{[*])}. $ Then $y\in ((A\circ
I)^{[*]})^{[\dagger]}\circ K^{[*]}+\mathcal{N}((A\circ I)^{[*]})=
(A^{[\dagger]})^{[*]}\circ I\circ K^{[*]}+\mathcal{N}((A\circ
I)^{[*]})$. This proves the first part.

 Next, suppose that
$A^{[\dagger]}\circ A\circ K\subseteq K$. Let $u=u^{1}+u^{2}$, where
$u^{1}=(A^{[\dagger]})^{[*]}\circ I\circ l$ with $l\in K^{[*]}$ and
$u^{2}\in \mathcal{N}((A\circ I)^{[*]}). $ Let $v=A\circ I\circ t$,
$t\in K$ and set $t^{'}=A^{[\dagger]}\circ A\circ t\in K$. Then $[u,
v]=[u^{1}+u^{2}, v]=[u^{1}, v]+[u^{2}, v]=[u^{1}, A\circ I\circ
t]=[(A^{[\dagger]})^{[*]}\circ I\circ l, A\circ I\circ t]=[l,
t^{'}]\geq 0$, since $[u^{2}, v]=[u^{2}, A\circ I\circ t]=0$. Thus
$u\in (A\circ I\circ K)^{[*]}$.
\end{proof}

\begin{rem}\label{remarkI}
The following example shows that in the absence of the condition
$A^{[\dagger]}\circ A\circ K\subseteq K$, the reverse inclusion may
not hold in Theorem \ref{dual of Aok}.
Let $A = \begin{pmatrix} 1 & 0 & 0\\
                     0 & -1 & 1

      \end{pmatrix}, $
      $M=\begin{pmatrix} 1 &  0\\
                         0 & -1
      \end{pmatrix} $ and
 $N =\begin{pmatrix} 1 & 0 & 0\\
                      0 &-1 & 0\\
                      0 & 0 &-1
      \end{pmatrix}$.
 Then $A^{\dagger} =\displaystyle\frac{1}{2}\begin{pmatrix} 2 & 0 \\
                                        0 & -1 \\
                                        0 & 1
      \end{pmatrix} $ and
  $A^{[\dagger]}=NA^{\dagger}M =\displaystyle\frac{1}{2}\begin{pmatrix} 2 & 0 \\
                                        0 & -1 \\
                                        0 & 1 \\
      \end{pmatrix}$.
  Let $K=\mathbb{R}^{3}_{+}$ then $K^{[*]}=N\mathbb{R}^{3}_{+}$.
  Suppose $x=(1,2,3)^t$. Then $A^{[\dagger]}\circ A\circ x=\displaystyle(1,\frac{-1}{2},\frac{1}{2})\notin K.$
So,  $A^{[\dagger]}\circ A\circ K\nsubseteq K$. Also $A\circ I = \begin{pmatrix} 1 & 0 & 0\\
                     0 & 1 & -1

      \end{pmatrix}$. So $\mathcal{N}((A\circ I)^{[*]})$ contins only the zero
      vector. Let $y=(1,2,0)^t\in K$ then $y^{1}=Ny=(1,-2,0)\in
      K^{[*]}$. Then $u=(A^{[\dagger]})^{[*]}\circ I\circ y^{1}=(1,1)^t\in
      (A^{[\dagger]})^{[*]}\circ I\circ K^{[*]}$. But $u\notin (A\circ I\circ K)^{[*]}$,
      since $[u, A\circ I\circ v]=\langle u, MAv \rangle<0$ for $v=(1,4,8)^t$.
 \end{rem}

The next result is analogous to Theorem \ref{dual of Aok}. This will
be used later.

\begin{lem}\label{analogus}
$ ((A^{[\dagger]})^{[*]}\circ I\circ K^{[*]})^{[*]}\subseteq A\circ
I\circ K+\mathcal{N}((A\circ I)^{[*]})$. If $A^{[\dagger]}\circ
A\circ K\subseteq K$, then equality holds.
\end{lem}
\begin{proof}
The proof follows from Lemma \ref{AKdualAK} and Theorem \ref{dual of
Aok} by replacing $A$ by $((A^{[\dagger]})^{[*]}$ and $K$ by
$K^{[*]}$.
%Let $u\in((A^{[\dagger]})^{[*]}\circ I\circ K^{[*]})^{[*]}$ and $r\in K^{[*]}$
%then $(A^{[\dagger]})^{[*]}\circ I\circ r\in (A^{[\dagger]})^{[*]}\circ I\circ K^{[*]}$
%$0\leq[u, (A^{[\dagger]})^{[*]}\circ I\circ r]=[A^{[\dagger]}\circ I\circ u, r]$ then
%$A^{[\dagger]}\circ I\circ u\in (K^{[*]})^{[*]}=K$ Thus $y\in
%((A^{[\dagger]})^{[*]}\circ I\circ K^{[*]})^{[*]}$ then
%$z=A^{[\dagger]}\circ I\circ y\in K$ Then
%$y=(A^{[\dagger]}\circ I)^{[\dagger]}\circ z+w$ where $w\in
%\mathcal{N}(A^{[\dagger]}\circ I)=\mathcal{N}(A\circ I)^{[\dagger]}=\mathcal{N}((A\circ I)^{[*]})$
%Then $y=A\circ I\circ z\in A\circ I\circ K+\mathcal{N}((A\circ I)^{[*]})$ By Lemma $3. 2$
%$A^{[\dagger]}\circ A\circ K\subseteq K \Leftrightarr\circ w
%(A\circ I^{[\dagger]}\circ A\circ I\circ K^{[*]})\subseteq K^{[*]}$ Let $u=u^{1}+u^{2}$
%where $u^{1}=A\circ I\circ l$ with $l\in K, u^{2}\in \mathcal{N}((A\circ I)^{[*]})$
%$v=(A^{[\dagger]})^{[*]}\circ I\circ t, t\in K^{[*]}$ set $t^{1}=
%A\circ I^{[\dagger]}\circ A\circ I\circ t\in K^{[*]}$
%$[u, v]=[u^{1}+u^{2}, v]=[u^{1}, v]+[u^{2}, v]$\\
%Then $[u, v]=[A\circ I\circ l, (A^{[\dagger]})^{[*]}\circ I\circ t]+[u^{2}, (A^{[\dagger]})^{[*]}\circ I\circ t]\\
%           =[l, I\circ A^{[*]}\circ (A^{[\dagger]})^{[*]}\circ I\circ t]+[A^{[\dagger]}\circ I\circ u^{2}, t]\\
%           =[l, I\circ A^{[\dagger]}\circ A\circ I\circ t]=[l, t^{1}]\geq0$\\
%           Thus $u\in ((A^{[\dagger]})^{[*]}\circ I\circ K^{[*]})^{[*]}$\\
\end{proof}
\begin{rem}
Let A be given as in Remark \ref{remarkI}. Then $A^{[\dagger]}\circ
A\circ K\nsubseteq K$. Let $y=(2,5,8)^t\in K $ and set
       $y^{1}=A\circ I\circ y=Ay=(2,3)^t\in A\circ I\circ K$. Let
       $v=N(1,2,0)^t=(1,-2,0)^t\in K^{[*]}$ and
       $z=(A^{[\dagger]})^{[*]}\circ I\circ v=
       (1,1)^t\in(A^{[\dagger]})^{[*]}\circ I\circ K^{[*]}.$ Then
   $[y^{1}, z]=\langle y^{1}, Mz \rangle=\langle (2, 3)^t, (1, -1)^t
   \rangle<0$,
   so that $y^{1}\notin ((A^{[\dagger]})^{[*]}\circ I\circ K^{[*]})^{[*]}$.
   This  shows that the condition $A^{[\dagger]}\circ A\circ K\subseteq K$ is essential for the reverse inclusion to hold in Lemma \ref{analogus}.
\end{rem}
\begin{lem}\label{intersection}
$(A\circ I\circ K)^{[*]}\cap \mathcal{R}(A\circ I)\subseteq
(A^{[\dagger]})^{[*]}oI\circ K^{[*]}$. If $A^{[\dagger]}\circ A\circ
K\subseteq K$, then equality holds in the above inclusion.
\end{lem}
\begin{proof}
Let $y=A\circ I\circ x\in (A\circ I\circ K)^{[*]}$. Then by Lemma
\ref{lemma2},
 $(A\circ I)^{[*]}\circ y\in K^{[*]}.$ Also,  $y=(A\circ I)\circ (A\circ I)^{[\dagger]}\circ y=((A\circ I)\circ (A\circ I)^{[\dagger]})^{[*]}\circ y
      =((A\circ I)^{[\dagger]})^{[*]}\circ (A\circ I)^{[*]}\circ y
      =(A^{[\dagger]})^{[*]}\circ I\circ (A\circ I)^{[*]}\circ y\in(A^{[\dagger]})^{[*]}\circ I\circ K^{[*]}$,
      proving that $(A\circ I\circ K)^{[*]}\cap \mathcal{R}(A\circ I)\subseteq (A^{[\dagger]})^{[*]}\circ I\circ K^{[*]}$.

 Conversely, suppose that $x\in (A^{[\dagger]})^{[*]}\circ I\circ K^{[*]}$. Then $x=((A\circ I)^{[\dagger]})^{[*]}\circ u$ for some $u\in K^{[*]}$.
 This implies $x\in \mathcal{R}(A\circ I)$. Let $w\in K$, $v=A\circ I\circ w\in
 A\circ I\circ K$ and $w^{1}=A^{[\dagger]}\circ A\circ w\in K$. Then we have $[x, v]=[(A^{[\dagger]})^{[*]}\circ I\circ u, A\circ I\circ w]=[u, A^{[\dagger]}\circ A\circ w]=[u,
 w^{1}]\geq0$.
 Thus $x\in (A\circ I\circ K)^{[*]}$.
\end{proof}
\begin{rem}
Let A be given as in Remark \ref{remarkI}. Then $A^{[\dagger]}\circ
A\circ K\nsubseteq K$. Let $y=(2,5,8)^t\in K$, and
       $y^{1}=A\circ I\circ y=Ay=(2,3)^t\in A\circ I\circ K$. Let
       $v=N(1,2,0)^t=(1,-2,0)^t\in K^{[*]}$ and
$z=(A^{[\dagger]})^{[*]}\circ I\circ v=(1,1)^t\in
(A^{[\dagger]})^{[*]}\circ I\circ K^{[*]}$. But
   $[y^{1}, z]=\langle y^{1}, Mz \rangle=\langle (2, 3)^t,(1,-1)^t \rangle<0.$
   Thus $z \notin (A\circ I\circ K)^{[*]}\cap \mathcal{R}(I\circ A)$.
   Hence the condition $A^{[\dagger]}\circ A\circ K\subseteq K$ is necessary for the reverse inclusion to hold in Lemma \ref{intersection}.
\end{rem}

Next, we obtain an equivalent condition for the acuteness of the
cone $(A\circ I\circ K)^{[*]}\cap\mathcal{R}(A\circ I)$.
\begin{lem}\label{acute}
Let $A^{[\dagger]}\circ A\circ K\subseteq K$. Then $(A\circ I\circ
K)^{[*]}\cap\mathcal{R}(A\circ I)$ is acute $\Leftrightarrow(A\circ
I\circ k)^{[*]}\cap\mathcal{R}(A\circ I)\subseteq A\circ I\circ K$.
\end{lem}
\begin{proof}
Suppose that $L=(A\circ I\circ K)^{[*]}\cap\mathcal{R}(A\circ I)$ is
acute. Then $L\subseteq L^{[*]}$. By Lemma \ref{analogus} and Lemma
\ref{intersection}, it follows that $L^{[*]}=((A\circ I\circ
K)^{[*]}\cap\mathcal{R}(A\circ I))^{[*]}
       =((A^{[\dagger]})^{[*]}\circ I\circ K^{[*]})^{[*]}=A\circ I\circ K+\mathcal{N}((A\circ I)^{[*]}).$
So, $(A\circ I\circ K)^{[*]}\cap\mathcal{R}(A\circ I)\subseteq
A\circ I\circ K+\mathcal{N}((A\circ I)^{[*]}).$  But, we have to
show that $(A\circ I\circ K)^{[*]}\cap\mathcal{R}(A\circ I)\subseteq
A\circ I\circ K.$ Let $x\in (A\circ I\circ
K)^{[*]}\cap\mathcal{R}(A\circ I)$. Then $x=A\circ I\circ u+z$, with
$u\in K$, $z\in \mathcal{N}((A\circ I)^{[*]})$. But since $x, A\circ
I\circ u\in \mathcal{R}(A\circ I)$, it follows that
 $z\in \mathcal{R}(A\circ I)\cap \mathcal{N}((A\circ I)^{[*]})=\{0\}.$
Thus $x\in A\circ I\circ K$.

 Conversely, let  $x,y\in
(A\circ I\circ K)^{[*]}\cap\mathcal{R}(A\circ I)\subseteq A\circ
I\circ K$. Then $x=A\circ I\circ u,~u\in K.$  We also have $(A\circ
I)^{[*]}\circ y\in K^{[*]}$. Now, $[x,y]=[A\circ I\circ
u,y]=[u,(A\circ I)^{[*]}\circ y]\geq 0.$ Thus $(A\circ I\circ
k)^{[*]}\cap\mathcal{R}(A\circ I)$ is acute.
\end{proof}
We next obtain a necessary and sufficient condition for a cone to be
invariant under $(A^{[*]}\circ A)^{[\dagger]}$ (See Lemma
\ref{nonnegativity}).
\begin{lem}\label{lemma 11}
$(A^{[\dagger]})^{[*]}\circ I\circ K^{[*]}\subseteq A\circ I\circ
K+\mathcal{N}((A\circ I)^{[*]})$ $\Leftrightarrow (A^{[*]}\circ
A)^{[\dagger]}\circ K^{[*]}\subseteq K+\mathcal{N}(A\circ I)$
\end{lem}
\begin{proof}
For $x\in K^{[*]}$, let $y=(A^{[*]}\circ A)^{[\dagger]}\circ
x=((A\circ I)^{[*]}\circ (A\circ I))^{[\dagger]}\circ x$ $= (A\circ
I)^{[\dagger]}\circ \linebreak((A\circ I)^{[\dagger]})^{[*]}\circ
x.$ Then
\begin{align*}
A\circ I\circ y &=(A\circ I)\circ (A\circ I)^{[\dagger]}\circ ((A\circ I)^{[\dagger]})^{[*]}\circ x\\
&=((A\circ I)^{[\dagger]}\circ (A\circ I)\circ (A\circ I)^{[\dagger]})^{[*]}\circ x\\
&=((A\circ I)^{[\dagger]})^{[*]}\circ x \\
&=(A^{[\dagger]})^{[*]}\circ I\circ x\in
(A^{[\dagger]})^{[*]}\circ I\circ K^{[*]}\\
 &\subseteq
A\circ I\circ K+\mathcal{N}((A\circ I)^{[*]})
\end{align*}
Therefore $A\circ I\circ y=A\circ I\circ v+w$, $v\in k$,
$w\in\mathcal{N}((A\circ I)^{[*]})$. So , $A\circ I\circ (y-v)\in
\mathcal{R}(A\circ I)\cap\mathcal{N}((A\circ I)^{[*]})=\{0\}$. Then
$A\circ I\circ (y-v)=0$. This implies, $ y-v=u\in\mathcal{N}(A\circ
I)$. Then $y=u+v$, $v\in K$, $u\in \mathcal{N}(A\circ I).$ This
shows that $(A^{[*]}\circ A)^{[\dagger]}\circ K^{[*]}\subseteq
K+\mathcal{N}(A\circ I)$.

Conversely, let $y=(A^{[\dagger]})^{[*]}\circ I\circ x,~x\in
K^{[*]}.$ Then $y=((A\circ I)^{[\dagger]})^{[*]}\circ x$ and
$(A\circ I)^{[\dagger]}\circ y=(A\circ I)^{[\dagger]}\circ ((A\circ
I)^{[\dagger]})^{[*]}\circ x=((A\circ I)^{[*]}\circ (A\circ
I))^{[\dagger]}\circ x$ $=(A^{[*]}\circ A)^{[\dagger]}\circ x=u+v$,
$u\in K,v\in \mathcal{N}(A\circ I).$ Then $y=((A\circ
I)^{[\dagger]})^{[\dagger]}\circ (u+v)+w$, $w\in \mathcal{N}((A\circ
I)^{[\dagger]})$ Then $y=A\circ I\circ u+w\in A\circ I\circ
K+\mathcal{N}((A\circ I)^{[*]})$.
\end{proof}
We also have stronger one-way implication, given below. The proof
follows from necessity  part of Lemma \ref{lemma 11}.
\begin{lem}\label{lemma12}
$(A^{[\dagger]})^{[*]}\circ I\circ K^{[*]}\subseteq A\circ I\circ K$
$\Rightarrow (A^{[*]}\circ A)^{[\dagger]}\circ K^{[*]}\subseteq
K+\mathcal{N}(A\circ I).$
\end{lem}
\begin{lem}\label{lemma13}
$(A^{[*]}\circ A)^{[\dagger]}\circ K^{[*]}\subseteq
K+\mathcal{N}(A\circ I)$ $\Rightarrow K^{[*]}\cap \mathcal{R}(A\circ
I)^{[*]}\subseteq A^{[*]}\circ A\circ K+\mathcal{N}((A\circ I)$
\end{lem}
\begin{proof}
Let $y=(A\circ I)^{[*]}\circ x\in K^{[*]}$. Then $(A^{[*]}\circ
A)^{[\dagger]}\circ y=u+z$, $u\in K$, $z\in \mathcal{N}(A\circ I).$
From this $y=(A^{[*]}\circ A)\circ (u+z)+w$, $w\in
\mathcal{N}(A^{[*]}\circ A)^{[\dagger]}.$ Since $A^{[*]}\circ
A=(A\circ I)^{[*]}\circ (A\circ I)$ and $z\in \mathcal{N}(A\circ
I)$, we get $y=A^{[*]}\circ A\circ u+w\in A^{[*]}\circ A\circ
K+\mathcal{N}(A\circ I)$
\end{proof}
\begin{lem}\label{nonnegativity}
Suppose that $A^{[\dagger]}\circ A\circ K\subseteq K$. Then
$(A^{[*]}\circ A)^{[\dagger]}\circ K^{[*]}\subseteq
K+\mathcal{N}(A\circ I)\Leftrightarrow (A^{[*]}\circ
A)^{[\dagger]}\circ K^{[*]}\subseteq K$
\end{lem}
\begin{proof}
It is enough to show the necessity part. Let $x\in K^{[*]}$ and
$y=(A^{[*]}\circ A)^{[\dagger]}\circ x.$ Then $(A^{[*]}\circ
A)^{[\dagger]}\circ x=u+v$ where $u\in K$, $v\in \mathcal{N}(A\circ
I)$. This implies $x=(A^{[*]}\circ A)\circ (u+v)+w,~w\in
\mathcal{N}(A\circ I)$, so that $y=(A^{[*]}\circ A)^{[\dagger]}\circ
(A^{[*]}\circ A)\circ u=A^{[\dagger]}\circ A\circ u\in K$.
\end{proof}
\begin{rem}
Let $A$ be as given in Remark \ref{remarkI} and let
$K=\mathbb{R}^{3}_{+}$. Then $K^{[*]}=NK^{*}=N\mathbb{R}^{3}_{+}$,
$A\circ I = \begin{pmatrix} 1 & 0 & 0\\
                     0 & 1 & -1

      \end{pmatrix}$,
$\mathcal{N}(A\circ I)=span\{(0,1,1)^t\}$  and
       $\mathcal{R}(A\circ I)^{[*]}=\{(x,-y,y)^t: x,y\in \mathbb{R}\}$.
       Also,
$ A^{[*]}\circ A=\begin{pmatrix} 1 & 0 & 0\\
                      0 &-1 & 1\\
                      0 & 1 &-1
      \end{pmatrix}$. So,
$(A^{[*]}\circ A)^{\dagger}=\displaystyle\frac{1}{4} \begin{pmatrix} 4 & 0 & 0\\
                      0 &-1 & 1\\
                      0 & 1 &-1
      \end{pmatrix}.$
Let $x^{1}=(x,y,z)^t\in K^*$,
  then $(A^{[*]}\circ A)^{[\dagger]}\circ Nx^{1}=\displaystyle\frac{1}{4}\begin{pmatrix} 4 x\\
                      -y+z \\
                     y-z
       \end{pmatrix}\in (A^{[*]}\circ A)^{[\dagger]}\circ K^{[*]}$.
       \newline
Since
$(4x,-y+z,y-z)^t=(4x,b,c)^t-\left(\displaystyle\frac{b+c}{2}\right)(0,1,1)^t$
where
 $b,c\geq0$ such that $\displaystyle\frac{b-c}{2}=-y+z$,nn  $(4x,-y+z,y-z)^t\in K+\mathcal{N}(A\circ I).$
Thus $(A^{[*]}\circ A)^{[\dagger]}\circ K^{[*]}\subseteq
K+\mathcal{N}(A\circ I)$.
  But for $x^{1}= (1,2,3)\in K^{*}$, $Nx^{1}\in K^{[*]}$
       and $(A^{[*]}\circ A)^{[\dagger]}\circ Nx^{1}=\displaystyle\frac{1}{4}(4,1,-1)^t\notin K$
       Thus $(A^{[*]}\circ A)^{[\dagger]}\circ K^{[*]}\nsubseteq K$. Hence we
       can conclude that in the absence of the condition $A^{[\dagger]}\circ A\circ K\subseteq
       K$, Lemma \ref{nonnegativity} may not be true.
\end{rem}

We are now in a position to prove the main result of this article.
\begin{thm}\label{main}(Main Result)
Let $A\in \mathbb{R}^{m\times n}$ with $\mathcal{R}(A\circ I)$
closed, K be a closed in $\mathbb{R}^{n}$ with $A^{[\dagger]}\circ
A\circ K\subseteq K$. Let $C=A\circ I\circ K$ and
$D=(A^{[\dagger]})^{[*]}\circ I\circ K^{[*]}.$ Then the following
conditions are equivalent:
\newline
(i) $D$ is acute.
\newline
(ii)  $(A^{[*]}\circ A)^{[\dagger]}\circ K^{[*]}\subseteq
K+\mathcal{N}(A\circ I)$.
\newline
(iii)  $C$ is obtuse.
\end{thm}
\begin{proof}
$(i)\Rightarrow(ii)$:
\newline
Suppose $D$ is acute then by definition, $D\subseteq D^{[*]}.$ By
Lemma \ref{analogus}, $D^{[*]}=A\circ I\circ K+\mathcal{N}(A\circ
I)^{[*]}.$ Thus $D\subseteq A\circ I\circ K+\mathcal{N}(A\circ
I)^{[*]}.$ Now, by Lemma \ref{lemma 11}, we obtain $(A^{[*]}\circ
A)^{[\dagger]}\circ K^{[*]}\subseteq
K+\mathcal{N}(A\circ I)$.\\
$(ii)\Rightarrow(i)$:
\newline
 Suppose
$(A^{[*]}\circ A)^{[\dagger]}\circ K^{[*]}\subseteq
K+\mathcal{N}(A\circ I).$ By Lemma \ref{lemma 11},  $D\subseteq
A\circ I\circ K+\mathcal{N}((A\circ I)^{[*]})$. But by Lemma
\ref{analogus}, $A\circ I\circ K+\mathcal{N}((A\circ
I)^{[*]})=D^{[*]}.$ So, $D\subseteq D^{[*]}. $ Hence $D$ is acute.
\newline$(ii)\Rightarrow(iii)$ Suppose
$(A^{[*]}\circ A)^{[\dagger]}\circ K^{[*]}\subseteq
K+\mathcal{N}(A\circ I).$ Note that $C=A\circ I\circ K$ is obtuse if
$C^{[*]}\cap\mathcal{R}(A\circ I)$ is acute. By Lemma \ref{acute},
it is enough to show that $C^{[*]}\cap\mathcal{R}(A\circ I)\subseteq
C.$

Let $y\in C^{[*]}\cap\mathcal{R}(A\circ I).$  Then $y=A\circ I\circ
x$ and by Lemma \ref{lemma2},  $(A\circ I)^{[*]}\circ y\in K^{[*]}$.
So, $(A\circ I)^{[*]}\circ y\in K^{[*]}\cap \mathcal{R}(A\circ
I)^{[*]}.$ By Lemma \ref{lemma13}, $(A\circ I)^{[*]}\circ
y=A^{[*]}\circ A\circ u+z$ with $u\in K$, $z\in \mathcal{N}(A\circ
I)$. Since $A^{[*]}\circ A=(A\circ I)^{[*]}\circ (A\circ I)$, it
follows that $(A\circ I)^{[*]}\circ y,A^{[*]}\circ A\circ u\in
\mathcal{R}(A\circ I)^{[*]}$. Thus $z\in \mathcal{R}(A\circ
I)^{[*]}\cap \mathcal{N}(A\circ I)=\{0\}$. This implies $z=0$. Then
$(A\circ I)^{[*]}\circ y=A^{[*]}\circ A\circ u$. From this,
\begin{align*}
y&=((A\circ I)^{[\dagger]})^{[*]}\circ ((A\circ I)^{[*]}\circ A\circ I\circ u)+w\\
&=((A\circ I)\circ (A\circ I)^{[\dagger]})^{[*]}\circ (A\circ I)\circ u+w\\
&= (A\circ I)\circ (A\circ I)^{[\dagger]}\circ (A\circ I)\circ u+w\\
&=(A\circ I)\circ u+w,
 \end{align*}
 where $w\in \mathcal{N}((A\circ I)^{[*]})$.
 \newline Since $y\in
 \mathcal{R}(A\circ I)$, it follows that $w\in \mathcal{R}(A\circ I)\cap
 \mathcal{N}(A\circ I)^{[*]})
 =\{0\}$. Thus $y\in A\circ I\circ K=C$.\\
$(iii)\Rightarrow(ii)$: \newline Let $C=A\circ I\circ K$ be obtuse.
Then by definition,   $C^{[*]}\cap\mathcal{R}(A\circ I)\subseteq C$.
By Lemma \ref{intersection}, $(A^{[\dagger]})^{[*]}\circ I\circ
K^{[*]}\subseteq C$. Now  by Lemma \ref{lemma12}, $(A^{[*]}\circ
A)^{[\dagger]}\circ K^{[*]}\subseteq K+\mathcal{N}(A\circ I). $
\end{proof}
\begin{rem}\hfill
\begin{enumerate}
\item [(i)]
The following exmaple illustrates Thorem \ref{main}.
Let $A=\begin{pmatrix} 1 & 0 & 1\\
                      1 &0 & 1
                      \end{pmatrix}$,
 $M=\begin{pmatrix}  0 & 1\\
                      1 &0 \\
                      \end{pmatrix}$,
  $N=\begin{pmatrix} 0 & 0 & 1\\
                      0 &1 & 0\\
                      1 &0 & 0
                      \end{pmatrix}$ and $K=\mathbb{R}^{3}_{+}$.
Then $A^{\dagger}=\displaystyle\frac{1}{4}\begin{pmatrix} 1&1\\
                        0 & 0\\
                        1 & 1
                        \end{pmatrix}$,
     $A^{[\dagger]}=NA^{\dagger}M=\displaystyle{\frac{1}{4}\begin{pmatrix} 1&1\\
                        0 & 0\\
                        1 & 1
                        \end{pmatrix}}$ and
                        $K^{[*]}=N\mathbb{R}^{3}_{+}$.
Note that for $x^{1}=(x,y,z)^t\in K$, $A^{[\dagger]}\circ A\circ
x^{1}=A^{[\dagger]}Ax^{1}=\displaystyle\frac{1}{2} (x+z,0,x+z)^t\in
K$.
   Thus $A^{[\dagger]}0A\circ K\subseteq K$.
   And
               $(A^{[*]}\circ A)^{\dagger}=\displaystyle{\frac{1}{16}}\begin{pmatrix} 2 & 0 & 2\\
                         0 & 0 & 0\\
                         2 & 0 & 2
             \end{pmatrix}$. Therefore
  $(A^{[*]}\circ A)^{[\dagger]}\circ K^{[*]}=N(A^{[*]}\circ A)^{\dagger}NK^{[*]}\subseteq
  K$. Also one can easily verify that $C=A\circ I\circ K$ is obtuse and
  $D=(A^{[\dagger]})^{[*]}\circ I\circ K^{[*]}$ is acute.
\item [(ii)] Here, we show by an example that in the absense of the condition $A\circ I=I\circ A$, Theorem \ref{main} may not hold.
Let $ A=\begin{pmatrix} 0 & 1\\
                     0 & 1
      \end{pmatrix}$,
        $ M=\begin{pmatrix} 0 & 1\\
                     1 & 0 \\
      \end{pmatrix}=N$. Then clearly $A\circ I\neq I\circ A.$
Let $K=\{(x,0):x\geq0 \}$  then  $K^{*}=\{(x,y):x\geq0,y\in \mathbb{
R}\}$ and $K^{[*]}=\{(y,x):x\geq0,~y\in \mathbb{R}\}$. Also,
$A^{\dagger}=\displaystyle\frac{1}{2}\begin{pmatrix} 0 & 0\\
                     1 & 1 \\
      \end{pmatrix}$ and
$A^{[\dagger]}=\displaystyle\frac{1}{2}\begin{pmatrix} 1 & 1\\
                     0 & 0
      \end{pmatrix}$.
 Clearly $A^{[\dagger]}\circ A\circ K\subseteq K$ and
$D=\left\{(\frac{x}{2},\frac{x}{2}): x\geq 0\right\}$ is acute but
$(A^{[*]}\circ A)^{[\dagger]}\circ K^{[*]}\nsubseteq K$
where $(A^{[*]}\circ A)^{[\dagger]}=\displaystyle\frac{1}{4}\begin{pmatrix} 0 & 2\\
                     0 & 0 \\
      \end{pmatrix}.$
      \item [(iii)] For given $A\in \mathbb{R}^{m\times n}$,  Ramanathan and Sivakumar
\cite{krksgram} derived a set of necessary and sufficient conditions
for a cone to be invariant under  $(A^{[*]}\circ A)^{[\dagger]}$ in
terms of pairwise acuteness of cones $D$ and $I\circ D$ in
indefinite inner product space. We would like to remark here that
pairwise acuteness of $D$ and $I\circ D$ is same as acuteness of the
cone $D$ in usual inner product space.

\end{enumerate}
\end{rem}
\vspace{.5cm} \noindent \textbf{Acknowledgements}: We thank
Prof.K.C. Sivakumar for his valuable comments and suggestions to
improve this article.

%\newpage

\end{document}